\documentclass[12pt]{amsart}
\usepackage{amsmath}
\usepackage{amscd}
\usepackage{amssymb}
\usepackage{amsfonts}

\newtheorem{theorem}{Theorem}[section]

\newtheorem{proposition}[theorem]{Proposition}

\theoremstyle{definition}

\theoremstyle{remark}
\newtheorem{remark}[theorem]{Remark}
\numberwithin{equation}{section}

\begin{document}
\title[K\"ahler--Ricci shrinkers with CSC]
{K\"ahler--Ricci Shrinkers with Constant Scalar Curvature Are Rigid}

\author{Wenqi Li}
\address{School of Mathematical Sciences, Key Laboratory of MEA (Ministry of
Education) and   Shanghai Key Laboratory of PMMP, East China Normal University, Shanghai
200241, China}
\email{51255500054@stu.ecnu.edu.cn}
\author{Jianyu Ou}
\address{School of Mathematical Sciences, Xiamen University, Xiamen 361005, China}
\email{oujianyu@xmu.edu.cn}
\author{Guoqiang Wu}
\address{School of Science, Zhejiang Sci-Tech University, Hangzhou 310018, China}
\email{gqwu@zstu.edu.cn}

\author{Detang Zhou}
\address{Instituto de Matem\'atica e Estat\'istica, Universidade Federal Fluminense, S\~ao Domingos, Niter\'oi, RJ 24210-201, Brazil}
\email{zhoud@id.uff.br}
\thanks{}
\dedicatory{}
\date{\today}

\begin{abstract}
Let $(M^{2m},g,f,J)$ be a complete gradient K\"ahler--Ricci shrinker satisfying
\[
\operatorname{Ric}+\nabla^2 f=\frac12 g.
\]
We prove that if the scalar curvature is constant, then the soliton is rigid.
More precisely, there exists an integer $k\in\{0,\ldots,m\}$ such that
\[
R\equiv k
\]
and
\[
(M^{2m},g,J)\cong
\bigl(N^{2k}\times\mathbb C^{m-k},
 g_N+g_{\mathrm{Euc}},J_N\oplus J_0\bigr),
\]
where $(N^{2k},g_N,J_N)$ is K\"ahler--Einstein and
\[
\operatorname{Ric}_{g_N}=\frac12 g_N.
\]
As part of the proof, we establish a Riemannian rigidity criterion:
for a complete nonsteady gradient Ricci soliton with constant scalar curvature,
the condition $\mathcal{L}_{\nabla f}\operatorname{Ric}=0$ implies radial
flatness and hence rigidity.
\end{abstract}
\maketitle

\section{Introduction}

A K\"ahler Ricci shrinker \((M,g,f,J)\) is a complete shrinking gradient Ricci soliton whose metric is K\"ahler and whose gradient vector field \(\nabla f\) is holomorphic. These objects arise naturally as models for finite-time singularities of the K\"ahler--Ricci flow and serve as noncompact generalizations of positive K\"ahler--Einstein metrics.

K\"ahler--Ricci shrinker surfaces have recently been completely classified. The classification theorem states that every K\"ahler Ricci shrinker surface is biholomorphic-isometric to one of the following: a closed del Pezzo surface with its unique shrinker metric; the Gaussian soliton on \(\mathbb{C}^2\); the Feldman--Ilmanen--Knopf shrinker on the blowup of \(\mathbb{C}^2\); the standard cylinder \(\mathbb{P}^1 \times \mathbb{C}\); or the BCCD shrinker on the blowup of \(\mathbb{P}^1 \times \mathbb{C}\).  Earlier classification results under bounded curvature assumptions were completed by Conlon--Deruelle--Sun \cite{CDS} and Bamler--Cifarelli--Conlon--Deruelle \cite{BCCD}. Later, Li--Wang \cite{Li-Wang} proved that every K\"ahler--Ricci shrinker surface has bounded sectional curvature, thereby removing the curvature hypothesis.

In this paper, we focus on gradient shrinking Ricci solitons with constant scalar curvature, with particular emphasis on the K\"ahler case. Recall that in Petersen--Wylie \cite{Petersen-Wylie2}, a gradient Ricci soliton $(M, g)$ is said to be rigid if it is isometric to a quotient of ${N} \times \mathbb{R}^k$, the product soliton of an Einstein manifold ${N}$ of positive scalar curvature with the Gaussian soliton $\mathbb{R}^k$.

For complete shrinkers, Huai-Dong Cao conjectured that $(M^n, g, f)$ has constant scalar curvature if and only if it is rigid, i.e., a finite quotient of ${N}^k \times \mathbb{R}^{n-k}$ for some Einstein manifold ${N}$ of positive scalar curvature; see \cite{Cheng-Zhou}.

Regarding this conjecture, Fern\'{a}ndez-L\'{o}pez--Garc\'{i}a-R\'{i}o \cite{FR10} and Munteanu--Sesum \cite{Munteanu-Sesum} proved that $n$-dimensional complete gradient shrinking solitons with harmonic Weyl tensor are rigid. Catino--Mastrolia--Monticelli \cite{CMM17} showed that any gradient shrinking Ricci soliton with fourth-order divergence-free Weyl tensor is rigid.

Petersen and Wylie \cite{Petersen-Wylie2} proved that a complete gradient Ricci soliton is rigid if and only if it has constant scalar curvature and is radially flat, that is, the sectional curvature $K(\cdot, \nabla f) = 0$. Fern\'{a}ndez-L\'{o}pez and Garc\'{i}a-R\'{i}o \cite{FR16} proved that, for a gradient Ricci soliton with constant scalar curvature, rigidity is equivalent to constancy of the rank of the Ricci endomorphism. They also derived that the possible values of $R$ are $\{0, \lambda, \dots, (n-1)\lambda, n\lambda\}$.

Several years ago, Cheng and Zhou \cite{Cheng-Zhou} confirmed Cao's conjecture in dimension $n = 4$. For some progress on the higher dimensional case, see \cite{Li-Ou-Qu-Wu,Li-Ou-Qu-Wu2,Li-Ou-Qu-Wu curv,Li-Ou-Qu-Wu3}. Chen--Zhu \cite{Chen-Zhu} proved that a complete shrinking (or expanding)
gradient K\"ahler--Ricci soliton with harmonic Bochner tensor is rigid; in
particular, its universal cover splits holomorphically as a K\"ahler--Einstein
factor times a complex Euclidean factor. For more work on the geometry and
topology of K\"ahler--Ricci shrinkers, see \cite{Cao}, \cite{Munteanu-Wang2},
\cite{Wang-Zhu}, and \cite{Sun-Zhang}.

The main result of the paper is as follows.

\begin{theorem}\label{mainthm}
Let $(M^{2m},g,f,J)$ be a complete gradient K\"ahler--Ricci shrinker satisfying
\[
\operatorname{Ric}+\nabla^2 f=\frac12 g.
\]
Assume that the scalar curvature $R$ is constant. Then there exists an integer
$k\in\{0,\ldots,m\}$ such that
\[
R\equiv k,
\]
and $(M,g,J)$ is holomorphically and isometrically a product
\[
M^{2m}\cong N^{2k}\times\mathbb C^{m-k},
\]
where $(N^{2k},g_N,J_N)$ is a K\"ahler--Einstein manifold satisfying
\[
\operatorname{Ric}_{g_N}=\frac12 g_N.
\]
\end{theorem}
\begin{remark}
It was recently proved by Esparza \cite{Carlos Esparza} that every
shrinking gradient K\"ahler--Ricci soliton is simply connected.
We will not need this result here, since in the present setting
simple connectedness follows directly from the rigidity structure
and the K\"ahler geometry of the zero section.
\end{remark}
\begin{remark}
An analogous argument applies to complete gradient K\"ahler--Ricci expanders,
with the usual possibility of a quotient in the rigid splitting.
\end{remark}

While this manuscript was being prepared, Hor\'acio
\cite{Horacio2026} posted a preprint proving the corresponding rigidity
statement for complete nonsteady gradient K\"ahler--Ricci solitons with
constant scalar curvature.  That work also uses the condition
$\mathcal{L}_{\nabla f}\operatorname{Ric}=0$ and derives
$A^2=\lambda A$ for the Ricci endomorphism $A$.  Its final rigidity step
invokes Proposition~1.3(2) of Petersen--Wylie.  In Section~3 below, we give
a direct derivation of radial flatness from
$\mathcal{L}_{\nabla f}\operatorname{Ric}=0$ and $A^2=\lambda A$; in
particular, our argument does not use Proposition~1.3(2) of
Petersen--Wylie.

The paper is organized as follows. In Section~2, we prove that a K\"ahler--Ricci
shrinker with constant scalar curvature satisfies
$\mathcal{L}_{\nabla f}\operatorname{Ric}=0$. In Section~3, we establish a
Riemannian rigidity criterion for complete nonsteady gradient Ricci solitons
with constant scalar curvature by proving radial flatness directly. Combining
these two results proves Theorem~\ref{mainthm}.

{\bf{Disclosure on AI assistance.}} The research problem and the principal mathematical ideas were formulated by the authors. The authors used AI-assisted tools,
principally ChatGPT 5.6 Sol to assist with checking some computations. The authors verified and completed all mathematical arguments and take
full responsibility for the content.

\section{K\"ahler--Ricci shrinkers with constant scalar curvature}

Let $(M^{2m},g,f,J)$ be a gradient K\"ahler--Ricci shrinker satisfying
\begin{equation}\label{eq:shrinker}
\operatorname{Ric}+\nabla^2 f=\frac12 g.
\end{equation}
We denote by $\rho$ the Ricci form,
\[
\rho(X,Y)=\operatorname{Ric}(JX,Y).
\]
Since $(M,g,J)$ is K\"ahler, the Ricci form is closed:
\[
d\rho=0.
\]

Recall Cartan's formula
\[
\mathcal{L}_X\alpha=d(\iota_X\alpha)+\iota_Xd\alpha
\]
for any differential form $\alpha$.

\begin{theorem}\label{thm:lie-ricci}
Let $(M^{2m},g,f,J)$ be a gradient K\"ahler--Ricci shrinker satisfying
\eqref{eq:shrinker}. If the scalar curvature is constant, then
\[
\mathcal{L}_{\nabla f}\operatorname{Ric}=0.
\]
\end{theorem}

\begin{proof}
For a gradient Ricci soliton satisfying \eqref{eq:shrinker}, one has
\begin{equation}\label{eq:gradR}
\nabla R=2\operatorname{Ric}(\nabla f,\cdot).
\end{equation}
Since $R$ is constant,
\[
\operatorname{Ric}(\nabla f,\cdot)=0.
\]

Using the $J$-invariance of the Ricci tensor, for every vector field $Y$,
\[
\operatorname{Ric}(J\nabla f,Y)
=-\operatorname{Ric}(\nabla f,JY).
\]
Consequently,
\[
(\iota_{\nabla f}\rho)(Y)
=\rho(\nabla f,Y)
=\operatorname{Ric}(J\nabla f,Y)
=-\operatorname{Ric}(\nabla f,JY)
=0.
\]
Thus $\iota_{\nabla f}\rho=0$. Since $d\rho=0$, Cartan's formula gives
\begin{equation}\label{eq:lie-rho}
\mathcal{L}_{\nabla f}\rho=0.
\end{equation}

We next observe that
\begin{equation}\label{eq:lie-J}
\mathcal{L}_{\nabla f}J=0.
\end{equation}
Indeed, since $\nabla J=0$, for every vector field $Y$,
\begin{align*}
(\mathcal{L}_{\nabla f}J)(Y)
&=[\nabla f,JY]-J[\nabla f,Y]\\
&=-\nabla_{JY}\nabla f+J\nabla_Y\nabla f.
\end{align*}
Let $A$ be the Ricci endomorphism defined by
\[
g(AX,Y)=\operatorname{Ric}(X,Y).
\]
The soliton equation gives
\[
\nabla_Y\nabla f=\frac12Y-AY.
\]
Since the Ricci tensor is $J$-invariant, $AJ=JA$. Hence
\[
\nabla_{JY}\nabla f=J\nabla_Y\nabla f,
\]
which proves \eqref{eq:lie-J}.

Finally, since $\rho(X,Y)=\operatorname{Ric}(JX,Y)$,
\[
(\mathcal{L}_{\nabla f}\rho)(X,Y)
=(\mathcal{L}_{\nabla f}\operatorname{Ric})(JX,Y)
+\operatorname{Ric}((\mathcal{L}_{\nabla f}J)X,Y).
\]
Using \eqref{eq:lie-rho} and \eqref{eq:lie-J}, we conclude that
\[
(\mathcal{L}_{\nabla f}\operatorname{Ric})(JX,Y)=0
\]
for all $X,Y$. Since $J$ is invertible,
\[
\mathcal{L}_{\nabla f}\operatorname{Ric}=0.
\]
\end{proof}

\section{A Riemannian rigidity criterion}

In this section the argument is purely Riemannian.  We prove radial
flatness directly and therefore do not use Proposition~1.3 of
Petersen--Wylie \cite{Petersen-Wylie2}; only their rigidity
characterization \cite[Theorem~1.2]{Petersen-Wylie2} is used at the final
step.

\begin{theorem}\label{thm:riemannian-rigidity}
Let $(M^n,g,f,\lambda)$ be a complete nonsteady gradient Ricci soliton,
\[
\operatorname{Ric}+\nabla^2f=\lambda g,
\qquad \lambda\neq0.
\]
Assume that the scalar curvature $R$ is constant and
\[
\mathcal{L}_{\nabla f}\operatorname{Ric}=0.
\]
Let $A:TM\to TM$ denote the Ricci endomorphism. Then
\begin{equation}\label{eq:A-polynomial}
A(A-\lambda I)=0.
\end{equation}
Moreover,
\[
R(X,\nabla f)\nabla f=0
\]
for every vector field $X$. In particular, the soliton is radially flat
and hence rigid.
\end{theorem}

\begin{proof}
Set
\[
V=\nabla f,
\]
and define $A$ by
\[
\operatorname{Ric}(X,Y)=g(AX,Y).
\]
The soliton equation gives
\begin{equation}\label{eq:hess-general}
\nabla_XV=(\lambda I-A)X.
\end{equation}

For arbitrary vector fields $X,Y$, the assumption
$\mathcal{L}_{V}\operatorname{Ric}=0$ gives
\begin{align*}
0
&=(\mathcal{L}_{V}\operatorname{Ric})(X,Y)\\
&=(\nabla_V\operatorname{Ric})(X,Y)
 +\operatorname{Ric}(\nabla_XV,Y)
 +\operatorname{Ric}(X,\nabla_YV).
\end{align*}
Using \eqref{eq:hess-general} and the self-adjointness of $A$, we obtain
\begin{equation}\label{eq:nablaVA-general}
\nabla_VA=2(A^2-\lambda A)
       =2A(A-\lambda I).
\end{equation}

For a gradient Ricci soliton one has
\begin{equation}\label{eq:scalar-drift-general}
\Delta_fR=2\lambda R-2|\operatorname{Ric}|^2.
\end{equation}
Since $R$ is constant,
\begin{equation}\label{eq:p2-general}
\operatorname{tr}(A^2)=|\operatorname{Ric}|^2=\lambda R.
\end{equation}
Thus $\operatorname{tr}(A^2)$ is constant. Differentiating in the
$V$-direction and using \eqref{eq:nablaVA-general}, we find
\begin{align*}
0
&=V\bigl(\operatorname{tr}(A^2)\bigr)\\
&=2\operatorname{tr}(A\nabla_VA)\\
&=4\left(\operatorname{tr}(A^3)
       -\lambda\operatorname{tr}(A^2)\right).
\end{align*}
Hence
\begin{equation}\label{eq:p3-general}
\operatorname{tr}(A^3)
=
\lambda\operatorname{tr}(A^2).
\end{equation}
In particular, $\operatorname{tr}(A^3)$ is constant. Differentiating
again,
\begin{align*}
0
&=V\bigl(\operatorname{tr}(A^3)\bigr)\\
&=3\operatorname{tr}(A^2\nabla_VA)\\
&=6\left(\operatorname{tr}(A^4)
       -\lambda\operatorname{tr}(A^3)\right),
\end{align*}
and therefore
\begin{equation}\label{eq:p4-general}
\operatorname{tr}(A^4)
=
\lambda\operatorname{tr}(A^3).
\end{equation}

Combining \eqref{eq:p3-general} and \eqref{eq:p4-general}, we obtain
\begin{align*}
|A(A-\lambda I)|^2
&=\operatorname{tr}(A^4)
  -2\lambda\operatorname{tr}(A^3)
  +\lambda^2\operatorname{tr}(A^2)\\
&=0.
\end{align*}
Consequently,
\[
A(A-\lambda I)=0.
\]
In particular, every eigenvalue of $A$ belongs to $\{0,\lambda\}$.
Equation \eqref{eq:nablaVA-general} now immediately gives
\begin{equation}\label{eq:nablaVAzero-general}
\nabla_VA=0.
\end{equation}

It remains to prove radial flatness. Set
\[
S=\nabla V=\nabla^2f=\lambda I-A.
\]
We use the curvature convention
\[
R(X,Y)Z
=
\nabla_X\nabla_YZ-\nabla_Y\nabla_XZ-\nabla_{[X,Y]}Z.
\]
For the associated $(0,4)$-tensor we write
\[
R(X,Y,Z,W)=g(R(X,Y)W,Z),
\]
so that, for orthonormal $X,Y$,
\[
\operatorname{Sec}(X,Y)=R(X,Y,X,Y).
\]
For the $(1,1)$-tensor $S=\nabla V$, the Ricci commutation identity gives
\begin{equation}\label{eq:hessian-commutator-general}
R(X,Y)V=(\nabla_XS)(Y)-(\nabla_YS)(X).
\end{equation}

Since $R$ is constant, the standard soliton identity
\[
\nabla R=2\operatorname{Ric}(V,\cdot)
\]
implies
\begin{equation}\label{eq:AVzero-general}
A(V)=0.
\end{equation}
Differentiating \eqref{eq:AVzero-general} in the direction $X$ gives
\[
0=(\nabla_XA)(V)+A(\nabla_XV).
\]
By \eqref{eq:hess-general} and $A(A-\lambda I)=0$,
\[
A(\nabla_XV)=A(\lambda I-A)X=0.
\]
Hence
\begin{equation}\label{eq:nablaXA-V-zero-general}
(\nabla_XA)(V)=0.
\end{equation}
On the other hand, we already know from
\eqref{eq:nablaVAzero-general} that
\[
\nabla_VA=0.
\]

Taking $Y=V$ in \eqref{eq:hessian-commutator-general} and using
$S=\lambda I-A$, we obtain
\begin{align*}
R(X,V)V
&=(\nabla_XS)(V)-(\nabla_VS)(X)\\
&=-(\nabla_XA)(V)+(\nabla_VA)(X)\\
&=0.
\end{align*}
Thus
\[
R(X,\nabla f)\nabla f=0
\]
for every vector field $X$, and the soliton is radially flat.

Since the scalar curvature is constant, Petersen--Wylie\'s rigidity
theorem \cite[Theorem~1.2]{Petersen-Wylie2} implies that the soliton is
rigid.
\end{proof}

\begin{proof}[Proof of Theorem~\ref{mainthm}]
By Theorem~\ref{thm:lie-ricci},
\[
\mathcal{L}_{\nabla f}\operatorname{Ric}=0.
\]
Applying Theorem~\ref{thm:riemannian-rigidity} with
$\lambda=\frac12$, we obtain
\[
A\left(A-\frac12I\right)=0
\]
and the soliton is rigid. Hence the Ricci eigenvalues are $0$ and
$\frac12$.

Since $A$ commutes with the complex structure $J$, both eigenspaces are
$J$-invariant and therefore have even real dimension. Write
\[
\operatorname{rank}A=2k,
\qquad 0\leq k\leq m.
\]
Then
\[
R=\operatorname{tr}A=k.
\]
Since
\[
A\left(A-\frac12I\right)=0,
\]
the soliton equation gives
\[
\nabla^2f
=
\frac12I-A
\geq0.
\]
Thus the potential function $f$ is convex.

By the rigidity theorem, $M$ is a flat vector bundle whose zero section
is an Einstein manifold $\Sigma^{2k}$ satisfying
\[
\operatorname{Ric}_{\Sigma}
=
\frac12g_{\Sigma},
\]
and, up to an additive constant, the potential function is
\[
f=\frac14|x|^2
\]
along the flat fibers. In particular,
\[
\Sigma=\{f=\min f\}.
\]

Since the Ricci endomorphism $A$ commutes with the complex structure
$J$, both its $0$- and $\frac12$-eigenspaces are $J$-invariant.
Consequently, the zero section $\Sigma$ is a K\"ahler submanifold and
\[
(\Sigma^{2k},g_{\Sigma},J_{\Sigma})
\]
is K\"ahler--Einstein with
\[
\operatorname{Ric}_{\Sigma}
=
\frac12g_{\Sigma}.
\]
Since $\operatorname{Ric}_{\Sigma}=\frac12 g_{\Sigma}>0$, the
Bonnet--Myers theorem implies that $\Sigma$ is compact. By a theorem of
Kobayashi \cite{Kobayashi1961}, every compact K\"ahler manifold with
positive definite Ricci tensor is simply connected. Hence $\Sigma$ is
simply connected.

On the other hand, the flat vector bundle $M$ deformation retracts
onto its zero section $\Sigma$. Hence
\[
\pi_1(M)\cong\pi_1(\Sigma)=0.
\]
Thus $M$ is simply connected.

Since a flat vector bundle over a simply connected base has trivial
holonomy, the flat bundle is trivial. Therefore
\[
M^{2m}
\cong
\Sigma^{2k}\times\mathbb R^{2(m-k)}.
\]
The $0$-eigenspace of $A$ is $J$-invariant, so the Euclidean factor
inherits its standard parallel complex structure. Hence
\[
\mathbb R^{2(m-k)}
\cong
\mathbb C^{m-k},
\]
and consequently
\[
(M^{2m},g,J)
\cong
\Sigma^{2k}\times\mathbb C^{m-k}
\]
holomorphically and isometrically. Renaming $\Sigma$ as $N$ gives
\[
\operatorname{Ric}_{g_N}=\frac12g_N,
\]
and completes the proof.
\end{proof}

\section{Further consequences}

The argument above also leads to a useful curvature identity when the Ricci
endomorphism satisfies the algebraic relation $A^2=\frac12A$, even without
assuming $\mathcal L_{\nabla f}\operatorname{Ric}=0$.

\begin{proposition}\label{prop:general-rank}
Let $(M^n,g,f)$ be a gradient shrinking Ricci soliton satisfying
\[
\operatorname{Ric}+\nabla^2f=\frac12g.
\]
Assume that the Ricci endomorphism $A$ satisfies
\[
A^2=\frac12A
\]
and has rank $k$. Then
\[
R=\frac{k}{2}.
\]
At any point choose an orthonormal Ricci eigenbasis
\[
\{e_1,\ldots,e_{n-k},e_{n-k+1},\ldots,e_n\}
\]
such that
\[
Ae_a=0,\qquad 1\leq a\leq n-k,
\]
and
\[
Ae_i=\frac12e_i,\qquad n-k+1\leq i\leq n.
\]
Writing
\[
K_{ij}=\operatorname{Sec}(e_i,e_j),
\]
one has
\begin{equation}\label{eq:partial-null-curvature}
|\nabla\operatorname{Ric}|^2
=
\sum_{1\leq a<b\leq n-k}K_{ab}.
\end{equation}
In particular, if the curvature contraction on the zero eigenspace satisfies
\[
\sum_{1\leq a<b\leq n-k}K_{ab}=0,
\]
then $\nabla\operatorname{Ric}=0$, and the soliton is rigid.
\end{proposition}

\begin{proof}
Since $A$ is self-adjoint and satisfies $A^2=\frac12A$, its eigenvalues are
$0$ and $\frac12$. Hence
\[
R=\operatorname{tr}A=\frac{k}{2},
\qquad
|\operatorname{Ric}|^2=\operatorname{tr}(A^2)=\frac{k}{4}.
\]
In particular, $|\operatorname{Ric}|^2$ is constant.

Set
\[
S_{00}=\sum_{1\leq a<b\leq n-k}K_{ab},
\qquad
S_{0+}=\sum_{\substack{1\leq a\leq n-k\\ n-k+1\leq i\leq n}}K_{ai},
\qquad
S_{++}=\sum_{n-k+1\leq i<j\leq n}K_{ij}.
\]
Summing the identities
\[
\operatorname{Ric}(e_a,e_a)=0,
\qquad 1\leq a\leq n-k,
\]
gives
\begin{equation}\label{eq:sum-zero-block}
2S_{00}+S_{0+}=0.
\end{equation}
Likewise, summing
\[
\operatorname{Ric}(e_i,e_i)=\frac12,
\qquad n-k+1\leq i\leq n,
\]
gives
\begin{equation}\label{eq:sum-positive-block}
\frac{k}{2}=S_{0+}+2S_{++}.
\end{equation}
Eliminating $S_{0+}$ from \eqref{eq:sum-zero-block} and
\eqref{eq:sum-positive-block}, we obtain
\begin{equation}\label{eq:Spp-S00}
S_{++}=\frac{k}{4}+S_{00}.
\end{equation}

For the normalization used here, the Bochner formula for the Ricci tensor is
\[
\frac12\Delta_f|\operatorname{Ric}|^2
=
|\nabla\operatorname{Ric}|^2
+|\operatorname{Ric}|^2
-2R_{ikjl}R_{ij}R_{kl}.
\]
In the chosen Ricci eigenbasis,
\begin{align*}
R_{ikjl}R_{ij}R_{kl}
&=\sum_{i,k}R_{ikik}\lambda_i\lambda_k\\
&=2\sum_{i<k}K_{ik}\lambda_i\lambda_k\\
&=\frac12S_{++}.
\end{align*}
Since $|\operatorname{Ric}|^2=k/4$ is constant, the Bochner formula and
\eqref{eq:Spp-S00} yield
\[
0
=
|\nabla\operatorname{Ric}|^2+\frac{k}{4}-S_{++}
=
|\nabla\operatorname{Ric}|^2-S_{00}.
\]
This proves \eqref{eq:partial-null-curvature}. If $S_{00}=0$, then
$\nabla\operatorname{Ric}=0$. The parallel splitting of the $0$- and
$\frac12$-eigenspaces gives the product decomposition on the universal cover;
the soliton equation shows that the potential is Gaussian on the flat factor.
Hence the soliton is rigid.
\end{proof}

The nullity-two case admits a particularly simple conclusion.

\begin{proposition}\label{prop:Rn-2}
Let $(M^n,g,f)$ be a gradient shrinking Ricci soliton satisfying
\[
\operatorname{Ric}+\nabla^2f=\frac12 g.
\]
Assume that the Ricci endomorphism has eigenvalues
\[
0,\ 0,\ \frac12,\ldots,\frac12.
\]
Then
\[
\nabla\operatorname{Ric}=0.
\]
In particular, the soliton is rigid.
\end{proposition}

\begin{proof}
The assumptions imply
\[
R=\frac{n-2}{2}
\]
and therefore, by the standard soliton identity
$\nabla R=2\operatorname{Ric}(\nabla f,\cdot)$,
\[
A(\nabla f)=0.
\]
At a regular point of $f$, choose an orthonormal Ricci eigenbasis
\[
e_1=\frac{\nabla f}{|\nabla f|},\qquad e_2,\ldots,e_n,
\]
such that
\[
Ae_1=Ae_2=0,
\qquad
Ae_i=\frac12e_i,\quad i=3,\ldots,n.
\]
Since $A^2=\frac12A$, differentiating this identity in the direction
$V=\nabla f$ and evaluating on $e_2$ gives
\begin{equation}\label{eq:diag-nablaV-A-null}
\langle(\nabla_VA)e_2,e_2\rangle=0.
\end{equation}

For a shrinking gradient Ricci soliton with constant scalar curvature,
Petersen--Wylie \cite[Lemma~2.5]{Petersen-Wylie2} gives
\begin{equation}\label{eq:radial-Ric-identity-nullity-two}
(\nabla_V\operatorname{Ric})(X,Y)
+g\left(A\left(\frac12I-A\right)X,Y\right)
=g(R(X,V)V,Y).
\end{equation}
Evaluating \eqref{eq:radial-Ric-identity-nullity-two} on $(e_2,e_2)$ and
using $A(A-\frac12I)=0$ together with
\eqref{eq:diag-nablaV-A-null}, we obtain
\[
g(R(e_2,V)V,e_2)=0.
\]
Thus
\[
K_{12}=0
\]
at every regular point. Proposition~\ref{prop:general-rank}, with
$k=n-2$, now gives
\[
|\nabla\operatorname{Ric}|^2=K_{12}=0
\]
on the regular set. The regular set is dense: otherwise $\nabla f$ would vanish on a
nonempty open set, so $\nabla^2f=0$ there and the soliton equation would give
$A=\frac12I$, contradicting the assumed zero eigenvalue. Hence, by continuity,
\[
\nabla\operatorname{Ric}=0
\]
on all of $M$. The soliton is therefore rigid.
\end{proof}

\begin{remark}\label{rem:Cheng-Zhou}
Proposition~\ref{prop:Rn-2} gives an alternative way to complete the final
rigidity step in the proof of Cheng--Zhou \cite{Cheng-Zhou}. In the normalized
four-dimensional case considered there, one has $R=1$. At the final stage of
their argument, Cheng and Zhou obtain that the Ricci tensor has rank two.
Together with $\operatorname{Ric}\geq0$, $R=1$, and
$|\operatorname{Ric}|^2=\frac12$, this yields the spectrum
\[
0,\quad 0,\quad \frac12,\quad \frac12.
\]
Once this rank-two spectral conclusion has been reached,
Proposition~\ref{prop:Rn-2} gives directly
\[
\nabla\operatorname{Ric}=0,
\]
and hence the rigid splitting. Thus the proposition provides an alternative
proof of the last geometric step once the rank-two conclusion of
Cheng--Zhou has been established.
\end{remark}

\bibliographystyle{amsplain}

\end{document}